\documentclass{llncs}
\usepackage{amssymb}
\usepackage{url}
\usepackage{graphicx}
\usepackage{amsfonts}
\usepackage{amsxtra}
\usepackage{algorithm}
\usepackage{algorithmic}
\usepackage{caption}
\usepackage{makeidx}

\begin{document}

\mainmatter

\title{One Mirror Descent Algorithm for Convex Constrained Optimization Problems with Non-Standard Growth Properties}

\titlerunning{One Mirror Descent Algorithm for Convex Constrained Optimization}
\author{Fedor S. Stonyakin\inst{1} \and Alexander A. Titov\inst{2}}
\authorrunning{Fedor Stonyakin and Alexander Titov}
\institute{
    V.\,I.\,Vernadsky Crimean Federal University, Simferopol\\
    \email{fedyor@mail.ru}
    \and
    Moscow Institute of Physics and Technologies, Moscow\\
    \email{a.a.titov@phystech.edu}
    }

\maketitle

\begin{abstract}
The paper is devoted to a special Mirror Descent algorithm for problems of convex minimization with functional constraints. The objective function may not satisfy the Lipschitz condition, but it must necessarily have a Lipshitz-continuous gradient. We assume, that the functional constraint can be non-smooth, but satisfying the Lipschitz condition. In particular, such functionals appear in the well-known Truss Topology Design problem. Also, we have applied the technique of restarts in the mentioned version of Mirror Descent for strongly convex problems. Some estimations for the rate of convergence are investigated for the Mirror Descent algorithms under consideration.
\end{abstract}

\section{Introduction}

The optimization of non-smooth functionals with constraints attracts widespread interest in large-scale optimization and its applications \cite{bib_ttd,bib_Shpirko}. There are various methods of solving this kind of optimization problems. Some examples of these methods are: bundle-level method \cite{bib_Nesterov}, penalty method \cite{bib_Vasilyev}, Lagrange multipliers method \cite{bib_Boyd}. Among them, Mirror Descent (MD) \cite{beck2003mirror,nemirovsky1983problem} is viewed as a simple method for non-smooth convex optimization.

In optimization problems with quadratic functionals we consider functionals which do not satisfy the usual Lipschitz property (or the Lipschitz constant is quite large), but they have a Lipshitz-continuous gradient. For such problems in (\cite{bib_Adaptive}, item 3.3) the ideas of \cite{bib_Nesterov,bib_Nesterov2016} were adopted to construct some adaptive version of Mirror Descent algorithm. For example, let $A_i$ ($i=1,\ldots,m$) be a positive-definite matrix: $x^TA_ix\geq 0\ \forall x$
and the objective function
$$f(x)=\max\limits_{1\leq i\leq m}f_i(x)$$
for
$$
f_i(x)=\frac{1}{2}\langle A_ix,x\rangle-\langle b_i,x\rangle+\alpha_i,\;i=1,\ldots,m.
$$
Note that such functionals appear in the Truss Topology Design problem with weights of the bars \cite{bib_Nesterov2016}.

In this paper we propose some partial adaptive (by objective functional) version of algorithm from (\cite{bib_Adaptive}, item 3.3). It simplifies work with problems where the necessity of calculating the norm of the subgradient of the functional constraint is burdensome in view of the large number of constraints. The idea of restarts \cite{bib_JuNe} is adopted to construct the proposed algorithm in the case of strongly convex objective and constraints. It is well-known that both considered methods are optimal in terms of the lower bounds \cite{nemirovsky1983problem}.

Note that a functional constraint, generally, can be non-smooth. That is why we consider subgradient methods. These methods have a long history starting with the method for deterministic unconstrained problems and Euclidean setting in \cite{shor1967generalized} and the generalization for constrained problems in \cite{polyak1967general}, where the idea of steps switching between the direction of subgradient of the objective and the direction of subgradient of the constraint was suggested. Non-Euclidean extension, usually referred to as Mirror Descent, originated in \cite{nemirovskii1979efficient,nemirovsky1983problem} and was later analyzed in \cite{beck2003mirror}. An extension for constrained problems was proposed in \cite{nemirovsky1983problem}, see also recent version in \cite{beck2010comirror}. To prove faster convergence rate of Mirror Descent for strongly convex objective in an unconstrained case, the restart technique \cite{nemirovskii1985optimal,nemirovsky1983problem,nesterov1983method} was used in \cite{bib_Juditsky}. Usually, the stepsize and stopping rule for Mirror Descent requires to know the Lipschitz constant of the objective function and constraint, if any. Adaptive stepsizes, which do not require this information, are considered in \cite{bib_Nemirovski} for problems without inequality constraints, and in \cite{beck2010comirror} for constrained problems.

We consider some Mirror Descent algorithms for constrained problems in the case of non-standard growth properties of objective functional
(it has a Lipshitz-continuous gradient).

The paper consists of Introduction and three main sections. In Section 2 we give some basic notation concerning convex optimization problems with functional constrains. In Section 3 we describe some partial adaptive version (Algorithm \ref{alg1}) of Mirror Descent algorithm from (\cite{bib_Adaptive}, item 3.3) and prove some estimates for the rate of convergence of Algorithm \ref{alg1}. The last Section 4 is focused on the strongly convex case with restarting Algorithm \ref{alg1} and corresponding theoretical estimates for the rate of convergence.

\section{Problem Statement and Standard Mirror Descent Basics}

Let $(E,||\cdot||)$ be a normed finite-dimensional vector space and $E^*$ be the conjugate space of $E$ with the norm:
$$||y||_*=\max\limits_x\{\langle y,x\rangle,||x||\leq1\},$$
where $\langle y,x\rangle$ is the value of the continuous linear functional $y$ at $x \in E$.

Let $X\subset E$ be a (simple) closed convex set. We consider two convex subdiffirentiable functionals $f$ and $g:X\rightarrow\mathbb{R}$. Also, we assume that $g$ is Lipschitz-continuous:

\begin{equation}\label{eq1}
|g(x)-g(y)|\leq M_g||x-y||\;\forall x,y\in X.
\end{equation}

We focus on the next type of convex optimization problems
\begin{equation}\label{eq2}
 f(x) \rightarrow \min\limits_{x\in \mathcal{X}},
\end{equation}
\begin{equation}
\label{problem_statement_g}
    \mathrm{s.t.} \hspace{0.3cm} g(x) \leq 0.
\end{equation}

Let $d : X \rightarrow \mathbb{R}$ be a distance generating function (d.g.f) which is continuously differentiable and $1$-strongly convex w.r.t. the norm $\lVert\cdot\rVert$, i.e.
$$\forall x, y, \in X \hspace{0.2cm} \langle \nabla d(x) - \nabla d(y), x-y \rangle \geq \lVert x-y \rVert^2,$$
and assume that $\min\limits_{x\in X} d(x) = d(0).$ Suppose we have a constant $\Theta_0$ such that
$d(x_{*}) \leq \Theta_0^2,$ where $x_*$ is a solution of (\ref{eq2}) -- (\ref{problem_statement_g}).

Note that if there is a set of optimal points $X_*$, then we may assume that
$$\min\limits_{x_* \in X_*} d(x_*) \leq \Theta_0^2.$$
For all $x, y\in X$ consider the corresponding Bregman divergence
$$V(x, y) = d(y) - d(x) - \langle \nabla d(x), y-x \rangle.$$
Standard proximal setups, i.e. Euclidean, entropy, $\ell_1/\ell_2$, simplex, nuclear norm, spectahedron can be found, e.g. in \cite{bib_Nemirovski}. Let us define the proximal mapping operator standardly
$$\mathrm{Mirr}_x (p) = \arg\min\limits_{u\in X} \big\{ \langle p, u \rangle + V(x, u) \big\} \; \text{ for each }x\in X \text{ and }p\in E^*.$$
We make the simplicity assumption, which means that $\mathrm{Mirr}_x (p)$ is easily computable.

\section{Some Mirror Descent Algorithm for the Type of Problems Under Consideration}

Following \cite{bib_Nesterov}, given a function $f$ for each subgradient $\nabla f(x)$ at a point $y \in X$, we define
\begin{equation}
v_f(x, y)=\left\{
\begin{aligned}
&\left\langle\frac{\nabla f(x)}{\|\nabla f(x)\|_{*}},x-y\right\rangle, \quad &\nabla f(x) \ne 0\\
&0 &\nabla f(x) = 0\\
\end{aligned}
\right.,\quad x \in X.
\label{eq:vfDef}
\end{equation}

In (\cite{bib_Adaptive}, item 3.3) the following adaptive Mirror Descent algorithm for Problem (\ref{eq2}) -- (\ref{problem_statement_g}) was proposed by the first author.

\begin{algorithm}
\caption{Adaptive Mirror Descent, Non-Standard Growth}
\label{alg01}
\begin{algorithmic}[1]
\REQUIRE $\varepsilon,\Theta_0^2,X,d(\cdot)$
\STATE $x^0=\underset{x\in X}{argmin}\,d(x)$
\STATE $I=:\emptyset$
\STATE $N\leftarrow 0$
\REPEAT
    \IF{$g(x^N)\leq\varepsilon \rightarrow$}
        \STATE $h_N\leftarrow\frac{\varepsilon}{||\nabla f(x^N)||_{*}}$
        \STATE $x^{N+1}\leftarrow Mirr_{x^N}(h_N\nabla f(x^N)) \; \text{("productive steps")}$
        \STATE $N\rightarrow I$
    \ELSE
        \STATE $(g(x^N)>\varepsilon)\rightarrow$
        \STATE $h_N\leftarrow\frac{\varepsilon}{||\nabla g(x^N)||_{*}^2}$
        \STATE $x^{N+1}\leftarrow Mirr_{x^N}(h_N\nabla g(x^N))\; \text{("non-productive steps")}$
    \ENDIF
    \STATE $N\leftarrow N+1$
\UNTIL{$\Theta_0^2 \leq \frac{\varepsilon^2}{2}|I| + \sum\limits_{k \not\in I} \frac{\varepsilon^2}{2||\nabla g(x^k)||_{*}^2}$}
\ENSURE $\bar{x}^N := \arg \min_{x^k, k\in I} f(x^k)$
\end{algorithmic}
\end{algorithm}

\newpage

For the previous method the next result was obtained in \cite{bib_Adaptive}.
\begin{theorem}\label{th01}
Let $\varepsilon > 0$ be a fixed positive number and
Algorithm \ref{alg01} works
\begin{equation}\label{eq08}
N=\left\lceil\frac{2\max\{1, M_g^2\}\Theta_0^2}{\varepsilon^2}\right\rceil
\end{equation}
steps. Then
\begin{equation}\label{eq09}
\min\limits_{k \in I} v_f(x^k,x_*)<\varepsilon.
\end{equation}
\end{theorem}

Let us remind one well-known statement (see, e.g. \cite{bib_Nemirovski}).

\begin{lemma}\label{lem1}
Let $f:X\rightarrow\mathbb{R}$ be a convex subdifferentiable function over the convex set $X$ and a sequence $\{x^k\}$ be defined by the following relation: $$x^{k+1}=Mirr_{x^k}(h_k\nabla f(x^k)).$$
Then for each $x\in X$
\begin{equation}\label{eq7}
h_k\langle\nabla f(x^k),x^k-x\rangle\leq\frac{h_k^2}{2}||\nabla f(x^k)||_*^2+V(x^k,x)-V(x^{k+1},x).
\end{equation}
\end{lemma}

\newpage

The following Algorithm \ref{alg1} is proposed by us for Problem (\ref{eq2}) -- (\ref{problem_statement_g}).

\begin{algorithm}
\caption{Partial Adaptive Version of Algorithm \ref{alg01}}
\label{alg1}
\begin{algorithmic}[1]
\REQUIRE $\varepsilon,\Theta_0^2,X,d(\cdot)$
\STATE $x^0=\underset{x\in X}{argmin}\,d(x)$
\STATE $I=:\emptyset$
\STATE $N\leftarrow 0$
    \IF{$g(x^N)\leq\varepsilon \rightarrow$}
        \STATE $h_N\leftarrow\frac{\varepsilon}{M_g \cdot ||\nabla f(x^N)||_{*}}$
        \STATE $x^{N+1}\leftarrow Mirr_{x^N}(h_N\nabla f(x^N)) \; \text{("productive steps")}$
        \STATE $N\rightarrow I$
    \ELSE
        \STATE $(g(x^N)>\varepsilon)\rightarrow$
        \STATE $h_N\leftarrow\frac{\varepsilon}{M_g^2}$
        \STATE $x^{N+1}\leftarrow Mirr_{x^N}(h_N\nabla g(x^N))\; \text{("non-productive steps")}$
    \ENDIF
    \STATE $N\leftarrow N+1$
    \ENSURE $\bar{x}^N := \arg \min_{x^k, k\in I} f(x^k)$
\end{algorithmic}
\end{algorithm}

Set $[N]=\{k\in\overline{0,N-1}\},\;J=[N]/I$, where $I$ is a collection of indexes of productive steps
\begin{equation}\label{eq4}
h_k=\frac{\varepsilon}{M_g||\nabla f(x^k)||_{*}},
\end{equation}
and $|I|$ is the number of "productive steps". Similarly, for "non-productive" steps from the set $J$ the analogous variable is defined as follows:
\begin{equation}\label{eq5}
h_k=\frac{\varepsilon}{M_g^2},
\end{equation}
and $|J|$ is the number of "non-productive steps". Obviously,
\begin{equation}\label{eq6}
|I|+|J|=N.
\end{equation}

Let us formulate the following analogue of Theorem \ref{th01}.
\begin{theorem}\label{th1}
Let $\varepsilon > 0$ be a fixed positive number and
Algorithm \ref{alg1} works
\begin{equation}\label{eq8}
N=\left\lceil\frac{2M_g^2\Theta_0^2}{\varepsilon^2}\right\rceil
\end{equation}
steps. Then
\begin{equation}\label{eq9}
\min\limits_{k\in I} v_f(x^k,x_*)<\frac{\varepsilon}{M_g}.
\end{equation}
\end{theorem}
\begin{proof}
1) For the productive steps from \eqref{eq7}, \eqref{eq4} one can get, that $$h_k\langle\nabla f(x^k), x^k-x\rangle\leq\frac{h_k^2}{2}||\nabla f(x^k)||_*^2+V(x^k,x)-V(x^{k+1},x).$$

Taking into account $\frac{h_k^2}{2}||\nabla f(x^k)||_*^2=\frac{\varepsilon^2}{2M_g^2}$, we have
\begin{equation}\label{eq001}
h_k\langle\nabla f(x^k),x^k-x\rangle=\frac{\varepsilon}{M_g}\left\langle\frac{\nabla f(x^k)}{||\nabla f(x^k)||_*},\, x^k-x\right\rangle=\frac{\varepsilon}{M_g}v_f(x^k,x).
\end{equation}

2) Similarly for the "non-productive" steps $k\in J$:
$$h_k(g(x^k)-g(x))\leq\frac{h_k^2}{2}||\nabla g(x^k)||_*^2+V(x^k,x)-V(x^{k+1},x).$$

Using (\ref{eq1}) and $||\nabla g(x)||\leq M_g$ we have
\begin{equation}\label{eq002}
h_k(g(x^k)-g(x))\leq\frac{\varepsilon^2}{2M_g^2}+V(x^k,x)-(x^{k+1},x).
\end{equation}

3) From (\ref{eq001}) and (\ref{eq002}) for $x=x_*$ we have:
$$
\frac{\varepsilon}{M_g}\sum\limits_{k\in I}v_f(x^k,x_*)+\sum\limits_{k\in J}\frac{\varepsilon}{M_g^2}(g(x^k)-g(x_*))\leq
$$
\begin{equation}\label{eq10}
\leq N\frac{\varepsilon^2}{2M_g^2}+
\sum\limits_{k=0}^{N-1}(V(x^k,x_*)-V(x^{k+1},x_*)).
\end{equation}

Let us note that for any $k \in J$
$$g(x^k)-g(x_*)\geq g(x^k)>\varepsilon$$
and in view of
$$\sum\limits_{k=1}^N(V(x^k,x_*)-V(x^{k+1},x_*))\leq\Theta_0^2$$
the inequality (\ref{eq10}) can be transformed in the following way:
$$\frac{\varepsilon}{M_g}\sum\limits_{k\in I}v_f(x^k,x_*)\leq N\frac{\varepsilon^2}{2M_g^2}+\Theta_0^2-\frac{\varepsilon^2}{M_g^2}|J|.$$

On the other hand,
$$\sum\limits_{k\in I}v_f(x^k,x_*)\geq|I|\min\limits_{k\in I}v_f(x^k,x_*).$$

Assume that
\begin{equation}\label{eq11}
\frac{\varepsilon^2}{2M_g^2}N\geq\Theta_0^2, \text{ or } N\geq\frac{2M_g\Theta_0^2}{\varepsilon^2}.
\end{equation}
Thus,
$$|I|\frac{\varepsilon}{M_g}\min v_f(x^k,x_*)<N\frac{\varepsilon^2}{2M_g^2}-\frac{\varepsilon^2}{M_g^2}|J|+\Theta_0^2\leq$$
$$\leq\frac{N\varepsilon^2}{M_g^2}-\frac{\varepsilon^2}{M_g^2}|J|=\frac{\varepsilon^2}{M_g^2}|I|,$$
whence
\begin{equation}\label{eq12}
|I|\frac{\varepsilon}{M_g}\min v_f(x^k,x_*)<\frac{\varepsilon^2}{M_g^2}|I|\Rightarrow\min v_f(x^k,x_*)<\frac{\varepsilon}{M_g}.
\end{equation}

To finish the proof we should demonstrate that $|I|\neq 0$.
Supposing the reverse we claim that $|I|=0\Rightarrow|J|=N$, i.e. all the steps are non-productive, so after using
$$g(x^k)-g(x_*)\geq g(x^k)>\varepsilon$$
we can see, that $$\sum\limits_{k=0}^{N-1} h_k(g(x^k)-g(x_*))\leq\sum\limits_{k=0}^{N-1}\frac{\varepsilon^2}{2M_g^2}+\Theta_0^2\leq
N\frac{\varepsilon^2}{2M_g^2}+N\frac{\varepsilon^2}{2M_g^2}=N\frac{\varepsilon^2}{M_g^2}.
$$
So, $$\frac{\varepsilon}{M_g^2}\sum\limits_{k=0}^{N-1}(g(x^k)-g(x_*))\leq\frac{N\varepsilon^2}{M_g^2}$$
and
$$N\varepsilon<\sum\limits_{k=0}^{N-1}(g(x^k)-g(x_*))\leq N\varepsilon.$$

So, we have the contradiction. It means that $|I|\neq 0$.
\end{proof}

The following auxiliary assertion (see, e.g \cite{bib_Nesterov,bib_Nesterov2016}) is fullfilled ($x_*$ is a solution of \eqref{eq2} --  \eqref{problem_statement_g}).
\begin{lemma}
Let us introduce the following function:
\begin{equation}\label{eq13}
\omega(\tau)=\max\limits_{x\in X}\{f(x)-f(x_*):||x-x_*||\leq\tau\},
\end{equation} where $\tau$ is a positive number.
Then for any $y \in X$
\begin{equation}\label{eq_lemma}
f(y) - f(x_*) \leq \omega(v_f(y,x_*)).
\end{equation}
\end{lemma}

Now we can show, how using the previous assertion and Theorem \ref{th1}, one can estimate the rate of convergence of the Algorithm \ref{alg1} if the objective function $f$ is differentiable and its gradient satisfies the Lipschitz condition:
\begin{equation}\label{eqlipgrad}
||\nabla f(x)-\nabla f(y)||_*\leq L||x-y|| \quad \forall x,y\in X.
\end{equation}

Using the next well-known fact
$$f(x)\leq f(x_*)+||\nabla f(x_*)||_*||x-x_*||+\frac{1}{2}L||x-x_*||^2,$$
we can get that
$$\min\limits_{k\in I}f(x^k)-f(x_*)\leq\min\limits_{k\in I} \left\{||\nabla f(x_*)||_*||x^k-x_*||+\frac{1}{2}L||x^k-x_*||^2\right\}.$$
So,
$$f(x)-f(x_*)\leq ||\nabla f(x_*)||_*\frac{\varepsilon}{M_g}+\frac{L\varepsilon^2}{2M_g}=\varepsilon_f+\frac{L\varepsilon^2}{2M_g},$$
where $$\varepsilon_f=||\nabla f(x_*)||_*\frac{\varepsilon}{M_g}.$$

That is why the following result holds.

\begin{corollary}\label{col3}
If $f$ is differentiable on $X$ and \eqref{eqlipgrad} holds. Then after $$N=\left\lceil\frac{2M_g^2\Theta_0^2}{\varepsilon^2}\right\rceil$$ steps of Algorithm \ref{alg1} working the next estimate can be fulfilled: $$\min\limits_{0\leq k\leq N}f(x^k)-f(x_*)\leq\varepsilon_f+\frac{L}{2}\frac{\varepsilon^2}{M_g^2},
$$ where $$\varepsilon_f=||\nabla f(x_*)||_*\frac{\varepsilon}{M_g}.$$
\end{corollary}

We can apply our method to some class of problems with a special class of non-smooth objective functionals.

\begin{corollary}\label{col31}
Assume that $f(x) = \max\limits_{i = \overline{1, m}} f_i(x)$, where $f_i$ is differentiable at each $x \in X$ and
$$
||\nabla f_i(x)-\nabla f_i(y)||_*\leq L_i||x-y|| \quad \forall x,y\in X.
$$
Then after $$N=\left\lceil\frac{2M_g^2\Theta_0^2}{\varepsilon^2}\right\rceil$$ steps of Algorithm \ref{alg1} working the next estimate can be fulfilled: $$\min\limits_{0\leq k\leq N}f(x^k)-f(x_*)\leq\varepsilon_f+\frac{L}{2}\frac{\varepsilon^2}{M_g^2},
$$ where $$\varepsilon_f=||\nabla f(x_*)||_*\frac{\varepsilon}{M_g}, \quad L = \max\limits_{i = \overline{1, m}} L_i.$$
\end{corollary}

\begin{remark}
Generally, $||\nabla f(x_*)||_* \neq 0$, because we consider some class of constrained problems.
\end{remark}

\section{On the Technique of Restarts in the Considered Version of Mirror Descent for Strongly Convex Problems}

In this subsection, we consider problem
\begin{equation}\label{eq14}
f(x)\rightarrow\min,\;\;g(x)\leq 0,\;\;x\in X
\end{equation}
with assumption \eqref{eq1} and additional assumption of strong convexity of $f$ and $g$ with the same parameter $\mu$, i.e.,
$$
f(y) \geq f(x) + \langle \nabla f(x), y-x \rangle + \frac{\mu}{2} \| y-x \|^2, \quad x, y \in X
$$
and the same holds for $g$. We also slightly modify assumptions on prox-function $d(x)$. Namely, we assume that $0 = \arg \min_{x \in X} d(x)$ and that $d$ is bounded on the unit ball in the chosen norm $\|\cdot\|_E$, that is
\begin{equation}
d(x) \leq \frac{\Omega}{2} \quad \forall x\in X : \|x \| \leq 1,
\label{eq:dUpBound}
\end{equation}
where $\Omega$ is some known number. Finally, we assume that we are given a starting point $x_0 \in X$ and a number $R_0 >0$ such that $\| x_0 - x_* \|^2 \leq R_0^2$.

To construct a method for solving problem \eqref{eq14} under stated assumptions, we use the idea of restarting Algorithm \ref{alg1}. The idea of restarting a method for convex problems to obtain faster rate of convergence for strongly convex problems dates back to 1980's, see e.g. \cite{nemirovsky1983problem,nesterov1983method}. To show that restarting algorithm is also possible for problems with inequality constraints, we rely on the following Lemma (see, e.g. \cite{bayandina2018primal-dual}).
\begin{lemma}\label{lem3}
If $f$ and $g$ are $\mu$-strongly convex functionals with respect to the norm $\|\cdot\|$ on $X$, $x_{\ast} = arg\min\limits_{x \in X} f(x)$, $g(x)\leq 0$ ($\forall x \in X$) and $\varepsilon_{f}>0$ and $\varepsilon_{g}>0$:
\begin{equation}\label{eq15}
f(x)-f(x_{\ast})\leq \varepsilon_{f},\;\;g(x)\leq\varepsilon_{g}.
\end{equation}
Then
\begin{equation}\label{eq16}
\frac{\mu}{2}\|x-x_{\ast}\|^{2}\leq\max\{\varepsilon_{f},\varepsilon_{g}\}.
\end{equation}
\end{lemma}
In conditions of Corollary \ref{col31}, after Algorithm \ref{alg1} stops the inequalities will be true \eqref{eq15} for $$\displaystyle{\varepsilon_{f}=\frac{\varepsilon}{M_g}\|\nabla f(x_{\ast})\|_{\ast}+\frac{\varepsilon^{2}L}{2M_g^2}}$$
and $\varepsilon_{g}=\varepsilon$. Consider the function $\tau: \mathbb{R}^{+}\rightarrow\mathbb{R}^{+}$:
$$\tau(\delta)=\max\left\{\delta\|\nabla f(x_{\ast})\|_{\ast}+\frac{\delta^{2}L}{2}; \; \delta M_g\right\}.$$
It is clear that $\tau$ increases and therefore for each $\varepsilon>0$ there exists $$\varphi(\varepsilon)>0:\;\;\tau(\varphi(\varepsilon))=\varepsilon.$$
\begin{remark}
$\varphi(\varepsilon)$ depends on $\|\nabla f(x_{\ast})\|_{\ast}$ and Lipschitz constant $L$ for $\nabla f$. If $\|\nabla f(x_{\ast})\|_{\ast} < M_g$ then $\varphi(\varepsilon) = \varepsilon$ for small enough $\varepsilon$:
$$
\varepsilon < \frac{2(M_g - \|\nabla f(x_{\ast})\|_{\ast})}{L}.
$$
For another case ($\|\nabla f(x_{\ast})\|_{\ast} > M_g$) we have $\forall \varepsilon > 0$:
$$
\varphi(\varepsilon) = \frac{\sqrt{\|\nabla f(x_{\ast})\|_{\ast}^2 + 2\varepsilon L} - \|\nabla f(x_{\ast})\|_{\ast}}{L}.
$$
\end{remark}
Let us consider the following Algorithm \ref{Alg2} for the problem \eqref{eq14}.

\begin{algorithm}[h!]
\caption{Algorithm for the Strongly Convex Problem}
\label{Alg2}
\begin{algorithmic}[1]
  \REQUIRE accuracy $\varepsilon > 0$; strong convexity parameter $\mu$; $\Theta_0^2$ s.t. $d(x) \leq \Theta_0^2  \quad  \forall x\in X : \|x \| \leq 1$; starting point $x_0$ and number $R_0$ s.t. $\| x_0 - x_* \|^2 \leq R_0^2$.
  \STATE  Set $d_0(x) = d\left(\frac{x-x_0}{R_0}\right)$.
	\STATE Set $p=1$.
	\REPEAT
    \STATE Set $R_p^2 = R_0^2 \cdot 2^{-p}$.
		\STATE Set $\varepsilon_p = \frac{\mu R_p^2}{2}$.
		\STATE Set $x_p$ as the output of Algorithm \ref{alg1} with accuracy $\varepsilon_p$, prox-function $d_{p-1}(\cdot)$ and $\Theta_0^2$.
		\STATE $d_p(x) \gets d\left(\frac{x - x_p}{R_p}\right)$.
    \STATE Set $p = p + 1$.
  \UNTIL{$p>\log_2 \frac{\mu R_0^2}{2\varepsilon}$.}
  \ENSURE $x_p$.
\end{algorithmic}
\end{algorithm}

The following theorem is fulfilled.
\begin{theorem}\label{th2}
Let $f$ and $g$ satisfy Corollary \ref{col31}. If $f, g$ are $\mu$-strongly convex functionals on $X\subset\mathbb{R}^{n}$ and $d(x)\leq \theta^{2}_{0}\;\;\forall\,x\in X,\;\;\|x\|\leq1$. Let the starting point $x_{0}\in X$ and the number $R_{0}>0$ be given and $\|x_{0}-x_{\ast}\|^{2}\leq R^{2}_{0}$. Then for $\displaystyle{\widehat{p}=\left\lceil\log_{2}\frac{\mu R_{0}^{2}}{2\varepsilon}\right\rceil}\;\;\;x_{\widehat{p}}$ is the $\varepsilon$-solution of Problem \eqref{eq2} -- \eqref{problem_statement_g}, where
$$\|x_{\widehat{p}}-x_{\ast}\|^{2}\leq\frac{2\varepsilon}{\mu}.$$
At the same time, the total number of iterations of Algorithm \ref{alg1} does not exceed
$$\widehat{p}+\sum_{p=1}^{\widehat{p}}\frac{2\theta^{2}_{0}Mg^{2}}{\varphi^{2}(\varepsilon_{p})},\;\;\text{where}\;\;\varepsilon_{p}=\frac{\mu R^{2}_{0}}{2^{p+1}}.$$
\end{theorem}
\begin{proof}
The function $d_{p}(x)$ ($p=0,1,2,\ldots$) is $1$-strongly convex with respect to the norm $\displaystyle\frac{\|\cdot\|}{R_{p}}$. By the method of mathematical induction, we show that $$\|x_{\widehat{p}}-x_{\ast}\|^{2}\leq R_{p}\;\;\forall p\geq0.$$
For $p=0$ this assertion is obvious by virtue of the choice of $x_{0}$ and $R_{0}$. Suppose that for some $p$: $\|x_{p}-x_{\ast}\|^{2}\leq R_{p}^{2}$. Let us prove that $\|x_{p+1}-x_{\ast}\|^{2}\leq R_{p+1}^{2}$. We have $d_{p}(x_{\ast})\leq\theta^{2}_{0}$ and on $(p+1)$-th restart after no more than $$\displaystyle\left\lceil\frac{2\theta^{2}_{0}M_g^{2}}{\varphi^{2}(\varepsilon_{p+1})}\right\rceil$$ iterations of Algorithm \ref{alg1} the following inequalities are true:
$$f(x_{p+1})-f(x_{\ast})\leq\varepsilon_{p+1},\;\;\;g(x_{p+1})\leq\varepsilon_{p+1}\;\text{ for }\;\varepsilon_{p+1}=\frac{\mu R^{2}_{p+1}}{2}.$$
Then, according to Lemma \ref{lem3}
$$\|x_{p+1}-x_{\ast}\|^{2}\leq\frac{2\varepsilon_{p+1}}{\mu}=R^{2}_{p+1}.$$
So, for all $p\geq0$
$$\|x_{p}-x_{\ast}\|^{2}\leq R^{2}_{p}=\frac{R^{2}_{0}}{2^{p}},\;\;\;f(x_{p})-f(x_{\ast})\leq\frac{\mu R^{2}_{0}}{2}2^{-p},\;\;\;g(x_{p})\leq\frac{\mu R^{2}_{0}}{2}2^{-p}.$$

For $\displaystyle{p=\widehat{p}=\left\lceil\log_{2}\frac{\mu R_{0}^{2}}{2\varepsilon}\right\rceil}$ the following relation is true
$$\|x_{p}-x_{\ast}\|^{2}\leq R^{2}_{p}=R^{2}_{0}\cdot2^{-p}\leq\frac{2\varepsilon}{\mu}.$$

It remains only to note that the number of iterations of the work of Algorithm \ref{alg1} is no more than
$$\sum_{p=1}^{\widehat{p}}\left\lceil\frac{2\theta^{2}_{0}M_g^{2}}{\varphi^{2}(\varepsilon_{p+1})}\right\rceil\leq\widehat{p}+\sum_{p=1}^{\widehat{p}}\frac{2\theta^{2}_{0}M_g^{2}}{\varphi^{2}(\varepsilon_{p+1})}.$$
\end{proof}

\subsubsection*{Acknowledgments.} The authors are very grateful to Yurii Nesterov, Alexander Gasnikov and Pavel Dvurechensky for fruitful discussions. The authors would like to thank the unknown reviewers for useful comments and suggestions.

The research by Fedor Stonyakin and Alexander Titov (Algorithm 3, Theorem 3, Corollary 1 and Corollary 2) presented was partially supported by Russian Foundation for Basic Research according to the research project 18-31-00219. The research by Fedor Stonyakin (Algorithm 2 and Theorem 2) presented was partially supported by the grant of the President of the Russian Federation for young candidates of sciences, project no. MK-176.2017.1.

\end{document}